\documentclass[11pt,leqno]{amsproc}

\usepackage{times,eucal,txfonts,dsfont,mathptmx,newtxmath,newtxtext}

\makeatletter
\def\msection{\@startsection{section} %name
	{1} % level
	{0pt} % indent
	{-1ex plus -.1ex minus -0.9ex} % beforeskip
	{-.9ex plus -.2ex} % afterskip
	{\bfseries} % style 
}
\def\msubsection{\@startsection{subsection} %name
	{2} % level
	{0pt} % indent
	{-1ex plus -.1ex minus -0.2ex} % beforeskip
	{-.9ex plus -.2ex} % afterskip
	{\normalfont} % style
} 
\makeatother

\usepackage[text={6.5in,9in},centering]{geometry}

\newcommand{\inv}{^{\raisebox{.2ex}{$\scriptscriptstyle-1$}}}

\newcommand{\cl}{\textit{c}\!\!\;\ell}  
         
\newtheorem{theorem}{Theorem}[section] 
\newtheorem{proposition}[theorem]{Proposition}
\newtheorem{lemma}[theorem]{Lemma}
\newtheorem{corollary}[theorem]{Corollary} 
\theoremstyle{definition}

\newtheorem{remark}[theorem]{Remark}    

\usepackage[dvipsnames]{xcolor}   
\usepackage{xparse}
\usepackage{xr-hyper}
\definecolor{brightmaroon}{rgb}{0.76, 0.13, 0.28}
\usepackage[linktocpage=true,colorlinks=true,hyperindex,citecolor=teal,linkcolor=Royal Blue]{hyperref}

\begin{document} 
  
\author{Amartya Goswami}

\address{[1] Department of Mathematics and Applied Mathematics, University of Johannesburg, South Africa.
\newline\hspace*{.35cm} [2] National Institute for Theoretical and Computational Sciences (NITheCS), Johannesburg, South Africa.}

\email{agoswami@uj.ac.za}

\title{Lower spaces of multiplicative lattices}

\date{}

%None of the above, but in this section;

%Spectra in general topology


\maketitle
   
\begin{abstract}
We consider some distinguished classes of elements of a multiplicative lattice endowed with  coarse lower topologies, and call them lower spaces. The primary objective of this paper is to study the topological properties of these lower spaces, encompassing lower separation axioms and compactness. We characterize lower spaces that exhibit sobriety. Introducing the concept of strongly disconnected spaces, we establish a correlation between strongly disconnected lower spaces and the presence of nontrivial idempotent elements in the corresponding multiplicative lattices. Additionally, we provide a sufficient condition for a lower space to be connected. We prove that the lower space of proper elements is a spectral space, we further explore continuous maps between lower spaces.
\end{abstract}   

\smallskip

\msection{Introduction and preliminaries}\label{intro}

Recall from \cite{WD39} (see also \cite{Dil62}) that
a \emph{multiplicative lattice} is a complete lattice $(L, \leqslant, 0, 1)$ endowed with a binary operation operation $\cdot  $, satisfying the following axioms:
\begin{enumerate}
		
\item\label{mla} $x\cdot   (y\cdot   z)=(x\cdot   y)\cdot   z$,
		
\item\label{mlc} $x\cdot   y = y\cdot   x$,
		
\item\label{jdp} $x\cdot   (\bigvee_{\lambda \in \Lambda} y_{\lambda})=\bigvee_{\lambda \in \Lambda}(x\cdot   y_{\lambda})$, 
		
\item \label{mid} $x\cdot   1=x$,
\end{enumerate}
for all $x,$ $y,$ $y_{\lambda}$, $z\in L$, and for all $\lambda \in \Lambda$, where $\Lambda$ is an index set. Note that  in the sense of \cite{Mul86},  a multiplicative lattice is a commutative, unital quantale.
We shall also use the notation $x^n$ to denote $x\cdot   \cdots \cdot   x$ (repeated $n$ times).
In the following lemma, we compile a number of elementary results on multiplicative lattices that  be  utilized in sequel.

\begin{lemma}\label{bip}
Let $L$ be a multiplicative lattice. Then the following hold.
\begin{enumerate}
\item \label{mul}
$x\cdot  y\leqslant x\wedge y$, for all $x, y\in L$.
		
\item $x\cdot  0=0$, for all $x\in L$.
		
\item\label{mon} If $x\leqslant y$, then $x\cdot  z\leqslant y\cdot   z,$ for all $z\in L$.
		
\item\label{monj} If $x\leqslant y$ and $u\leqslant v$, then $x\cdot  u\leqslant y\cdot  v$, for all $x,$ $y,$ $u,$ $v\in L$.
\end{enumerate}
\end{lemma} 

\begin{proof}
1. Notice that $x=x\cdot  1=x\cdot  (y\vee 1)=(x\cdot  y)\vee x,$ and this implies $x\cdot  y\leqslant x$. Similarly, we have $x\cdot  y\leqslant y,$ and hence $x\cdot  y\leqslant x\wedge y.$
		
2. Applying 1., we obtain $x\cdot  0\leqslant x\wedge 0=0,$ and hence $x\cdot  0=0.$
	
3. Since $x\leqslant y$, we have $y\cdot  z=(x\vee y)\cdot  z=(x\cdot  z)\vee (y\cdot  z)$, which implies that $x\cdot  z\leqslant y\cdot  z.$
	
4. Since $x\leqslant y$ and $u\leqslant v$, by 3., we have $x\cdot  u\leqslant y\cdot  u$ and $y\cdot  u\leqslant y\cdot  v$. From these, we obtain the claim.
\end{proof} 

Let us now focus on some distinguished classes of elements of a multiplicative lattice $L$. An element $x$ of $L$ is called \emph{proper} if $x\neq 1$, and by $\mathrm{Prop}(L)$, we denote the set of all proper elements of $L$.
A proper element $p$ of a multiplicative lattice  $L$ is said to be \emph{prime} if $x\cdot y\leqslant p$ implies that $x\leqslant p$ or $y\leqslant p$ for all $x$, $y\in L$. The set of prime elements of $L$ is denoted by $\mathrm{Spec}(L)$. An element $p$ is said to be a \emph{minimal prime element} if there is no prime element
$q$ such that $q \leqslant p.$ A proper element $m$ of $L$ is said to be \emph{maximal}, if $m \leqslant n$, $n\leqslant 1$, and $n\neq 1$  then $m = n.$ By $\mathrm{Max}(L)$, we denote the set of all maximal elements of $L$. An element $c$ of $L$ is called \emph{compact} if $c\leqslant \bigvee_{\lambda \in \Lambda} x_{\lambda}$ implies $c\leqslant \bigvee_{i=1}^n x_{\lambda_i}$, for some $n\in \mathds{N}^+$. A multiplicative lattice $L$ is said to be \emph{compactly generated} if every element of
$L$ is the join of compact elements. If $L$ is a compactly generated multiplicative lattice
with $1$ compact, then maximal elements exist in $L$, and every maximal element is a prime element (see \cite{AAJ95}). We say a multiplicative lattice $L$ satisfies the \emph{max-bounded} property if every proper element of $L$ is bounded by a maximal element of $L$. An element $x$ of $L$ is called \emph{idempotent} if $x\cdot x=x$. 

The \emph{Jacobson radical} of $L$ is defined as the set \[\mathrm{Jac}(L):=\bigwedge\left\{m\mid m\in \mathrm{Max}(L) \right\}.\]   The \emph{radical} of an element $x$ of $L$ is defined as \[\sqrt{x}:=\bigvee\left\{y\in L\mid\,y\,\text{is compact},\, y^n\leqslant x,\; n\in \mathds{N}^+\right\}=\bigwedge\left\{p\in \mathrm{Spec}(L)\mid x\leqslant p\right\}.\] An element $r$ of $L$ is called a \emph{radical element} if $r=\sqrt{r}$. By $\mathrm{Rad}(L)$, we denote the set of all radical elements of $L$. An element $x$ of $L$ is called \emph{nilpotent} if $x^n=0$, for some $n\in \mathds{N}^+$. The set of all nilpotent elements of a multiplicative lattice  $L$ is denoted
by $\mathrm{Nil}(L).$ An element $q\in L$ is said to be \emph{primary} if $x\cdot y\leqslant q$ implies $x\leqslant q$ or $y\leqslant \sqrt{q}$, for all $x$, $y\in L$. We denote the set of all primary elements of $L$ by $\mathrm{Prim}(L)$. 

An element $s\in L$ is said to be \emph{strongly irreducible} if $x\wedge y \leqslant s$ implies $x\leqslant s$ or $y\leqslant s$, for all $x,$ $y\in L$. We denote all  strongly irreducible elements of $L$ by $\mathrm{Irr}^+(L)$. By Lemma \ref{bip}(\ref{mul}), it follows that every prime element is strongly irreducible. Similarly, an element $s$ of $L$ is called \emph{irreducible} if $x\wedge y=s$ implies $x=s$ or $y=s$. We denote the set of all irreducible elements of $L$ by $\mathrm{Irr}(L)$.  An element $s$ is said to be \emph{completely irreducible} if $s=\bigwedge_{\lambda \in \Lambda}x_{\lambda}$ implies $s=x_{\lambda}$, for all $\lambda \in \Lambda$, and by $\mathrm{Irr}^{++}(L),$ we denote the set of all such elements of $L$.  
We use the symbol $\Sigma_L$ to denote any one of the above-mentioned distinguished classes of elements of a multiplicative lattice $L$.

From the perspective of abstract ideal theory, an in-depth exploration of certain classes of elements, along with others, is available in \cite{And74, AJ84, AAJ95, Dil62, Dil36, War37, WD39}. While these classes play a role in the examination of topological properties of multiplicative lattices, the consideration of a specific class (or classes) of elements as the underlying set of a topological space has garnered attention more recently.

The spectrum of prime elements of a multiplicative lattice has been extensively investigated in \cite{FFJ22} (see also \cite{Fac23}). This work provides a sufficient condition for a spectrum to be a spectral space under the Zariski topology. The paper also discusses some topological aspects of radical, solvable, locally solvable, and semiprime elements.

In the context of the multiplicative lattice 
$\mathcal{N}(G)$ of all normal subgroups of a group 
$G$, \cite{FGT23} endows the Zariski spectrum 
$\mathrm{Spec}(G)$ with the Zariski topology. It is proven that the topological space 
$\mathrm{Spec}(G)$ is spectral if and only if it is compact, if and only if 
$G$ has a maximal normal subgroup.

The objective of this paper is to examine topological properties of distinct classes of elements of a multiplicative lattice endowed with coarse lower topology. Notably, the topology associated with each class aligns with the Zariski or Stone topologies when the elements in question are prime or maximal, respectively. This framework enables the comprehensive investigation of topological properties for any class of elements of a multiplicative lattice. Furthermore, it extends and generalizes certain results presented in \cite{DG22} and \cite{FGS23}.

\smallskip

\msection{Lower spaces}

Let $L$ be a multiplicative lattice . The coarse lower topology (also known as lower topology) on a distinguish class $\Sigma_L$ of elements of $L$ will be the topology for which the sets of the type 
$$
v(x):=\{l\in \Sigma_L\mid x\leqslant l\}
$$
(where $x\in L$)  form a subbasis of closed sets.
We call a class $\Sigma_L$ of elements of $L$ endowed with the lower topology, a \emph{lower space}.  In general, the collection $\{v(x)\}_{x\in L}$ does not form a closed-set topology on a $\Sigma_L$. The following result in fact characterizes such class of elements of a multiplicative lattice. The proof  is straightforward.

\begin{theorem}\label{hkt}
The collection $\{v(x)\}_{x\in L}$ induces a closed-set topology on  $\Sigma_L$  if and only if the class  $\Sigma_L$ of elements satisfies the following property:    
\begin{equation}\label{hkp}
x\wedge  y\leqslant s\;\; \Rightarrow\;\; x\leqslant s\;\; \text{or}\;\; y\leqslant s,	
\end{equation}
for all $x,$ $y\in L $, and for all $s\in  \Sigma_L.$  
\end{theorem}  

It is clear from the above Theorem \ref{hkt} that the class of strongly irreducible (or \emph{meet-irreducible}) elements of a multiplicative lattice is the ``largest'' class of elements for which the collection $\{v(x)\}_{x\in L}$ will induce a closed-set topology. It is also evident that the classes of prime, minimal prime, maximal elements are all subclasses of the class of strongly irreducible elements.

\begin{remark}
If $v(-)$ is a Kuratowski closure operator on a class $\Sigma_L$ of elements in $L$, the collection $\{v(x)\}_{x\in L}$ of subsets of $\Sigma_L$ are exactly the closed subsets of the topology induced on $\Sigma_L$. However, if $v(-)$ is not a Kuratowski closure operator then the collection $\{v(x)\}_{x\in L}$ does not form a closed base of a topology on $\Sigma_L.$	
\end{remark}

\begin{proposition}\label{hrrad}
Let $L$ be a multiplicative lattice and $\Sigma_L$ be a class of elements of $L$. Then the following are equivalent.
\begin{enumerate}
\item 
\label{hahrada}
$v(x )=v(\sqrt{x })$ for every $x \in L$.

\item
$v(x )=v(\sqrt{x })$ for every $x \in \Sigma_L$. 
 
\item 
$\Sigma_L=\mathrm{Rad}(L)$. 
\end{enumerate}
\end{proposition}

\begin{proof}
1.$\Rightarrow$2. This is immediate because $\Sigma_L\subseteq L$.
	
2.$\Rightarrow$3. Let $x \in\Sigma_L$. Then $x \in v(x )$, and so, by 2, $x \in v(\sqrt{x })$. But this implies that $\sqrt{x }\leqslant x $, whence we have $x =\sqrt x $, showing that $x\in \mathrm{Rad}(L) $. 
	
3.$\Rightarrow$1. Let $x \in L$. If $v(x )=\emptyset$, then $v(\sqrt{x })=\emptyset$ because $v(\sqrt{x })\subseteq v(x )$, hence 	$v(x )=v(\sqrt{x })$. If $v(x )\neq \emptyset$, consider  any $y \in v(x )$. Then $y \in\Sigma_L$, and is therefore a radical element, by 3. Furthermore, $x \leqslant y $, which then implies that $\sqrt{x }\leqslant \sqrt{y }=y $, whence $y \in v(\sqrt{x })$. Therefore, $v(x )\subseteq v(\sqrt{x })$, and hence $v(x )=v(\sqrt{x })$, as required.  
\end{proof}  

From Proposition \ref{hrrad}, it is clear that  $v(x)=v(\sqrt{x})$ holds only for the radical elements of $L$. In the next result, we generalize the notion of radical of an element in such a fashion that the above identity  holds for all classes of elements of a multiplicative lattice.

\begin{proposition}\label{hrx} 
Suppose $L$ is a multiplicative lattice. Then for any class $\Sigma_L$ of elements of $L$ the following hold. 
\begin{enumerate}
\label{rxlr}  
\item\label{axaa} For every $x \in L,$ $x \leqslant \bigwedge v(x) .$
		
\item \label{axa}   
If $x \in\Sigma_L,$ then $x =\bigwedge v(x) $.
		
\item \label{hahxa} For all $x  \in L,$ $v(x )=v(\bigwedge v(x) ).$
		
\item\label{xbxa}  $v(x )\subseteq v(y )$ if and only if  $\bigwedge v(y) \subseteq\bigwedge v(x) $, for all $x ,$ $y \in L$.
		
\item\label{regc} If $L=\mathrm{Rad}(L),$ then  $v(x )\subseteq v(y )$ if and only if  $y \leqslant x $, for all $x ,$ $y \in L$.
\end{enumerate}
\end{proposition}

\begin{proof}
1.-2. Straightforward.
	
3. By 1., $x \leqslant \bigwedge v(x) ,$ and hence we have $v(x )\supseteq v(\bigwedge v(x) ).$ The other inclusion is obvious.
	
4. Suppose that $\bigwedge v(y)\subseteq \bigwedge v(x) $, and  $z\in v(x )$. Then $z\in \Sigma_L$ and $x \leqslant z$, so that $\bigwedge v(x) \leqslant z$. Therefore, 
\[
z\in v(z)\subseteq v\left(\bigwedge v(x) \right)\subseteq v\left(\bigwedge v(y)\right)\subseteq v(y ),
\]
showing that $v(x )\subseteq v(y )$. The other implication is obvious.
	
5. Since $L=\mathrm{Rad}(L),$ we have  $v(x )\subseteq v(y )$ implies $y \leqslant \bigwedge v(x )=x .$
The other implication follows from monotonicity of $v$.
\end{proof}

\begin{lemma}\label{lfc}
Let $L$ be a compactly generated multiplicative lattice and satisfies the ``max-bounded'' condition. A lower space $\Sigma_L$ have the property: $v(x)=\emptyset\Leftrightarrow x=L$ if and only if $\Sigma_L$ contains all maximal elements of $L$.
\end{lemma}

\begin{proof}
Suppose $\Sigma_L$ has the above property. Let $m$ be a maximal element of $R$. If $m\notin\Sigma_L$, then we would have $v(m) = \emptyset$, and since $m \neq L$, we would have a contradiction.
Conversely, suppose $\Sigma_L$ contains all maximal elements of $L$. Consider an element $x$ of $L$ such that $v(x) = \emptyset$. This means that there is no proper element of $L$ belonging to $\Sigma_L$ containing $x$, and this forces $x$ to be $L$ because every  proper element of $L$ is contained in some maximal element (this follows from the  ``max-bounded'' condition).
\end{proof}

\begin{proposition}\label{csb}
Let $L$ be a compactly generated multiplicative lattice with $1$ compact, and satisfies the max-bounded property. If a class $\Sigma_L$ of elements of $L$ contains the maximal elements of $L$, then the lower space $\Sigma_L$ is compact.
\end{proposition}

Note that the assumption on $L$ guarantees the existence of maximal elements. Since our topology is generated by subbasis closed sets, in the following proof we rely on Alexander subbasis theorem.

\begin{proof}
Suppose $\{C_{\lambda}\}_{\lambda \in \Lambda}$ is a collection of subbasis closed sets of $\Sigma_L$ with $\bigwedge_{\lambda \in \Lambda}C_{\lambda}=\emptyset.$ Then there exists $\{x_{\lambda}\}_{\lambda \in \Lambda}\subseteq L$ such that $v(x_{\lambda})=C_{\lambda}$, for $\lambda \in \Lambda$. This implies that
\[\emptyset=\bigwedge_{\lambda \in \Lambda}v(x_{\lambda})=v\left( \bigvee_{\lambda \in \Lambda} x_{\lambda} \right),\]
where the second equality indeed holds:
\[ l\in v\left( \bigvee_{\lambda \in \Lambda} x_{\lambda} \right) \Leftrightarrow \bigvee_{\lambda \in \Lambda} x_{\lambda}\leqslant l\Leftrightarrow (\lambda \in \Lambda) \; x_{\lambda}\leqslant l  \Leftrightarrow (\lambda \in \Lambda) \;l\in v(x_{\lambda})\Leftrightarrow l \in \bigwedge_{\lambda \in \Lambda} v(x_{\lambda}).\]
By Lemma \ref{lfc}, this implies $\bigvee_{\lambda \in \Lambda} x_{\lambda}=L.$  Since  $1$ is compact, there exists a finite subset $\{\lambda_i\}_{i=1}^n$ of $\Lambda$ such that $1=\bigvee_{i=1}^n x_{\lambda_i}.$ This proves the finite intersection property. Hence by Alexander subbasis theorem, we have the desired compactness.
\end{proof}

\begin{corollary}\label{ciqc}
Let $L$ be a compactly generated multiplicative lattice with $1$ compact, and satisfies the max-bounded property. Then the lower spaces $\mathrm{Spec}(L)$, $\mathrm{Max}(L)$, $\mathrm{Irr}(L)$, $\mathrm{Irr}^+(L)$, $\mathrm{Irr}^{++}(L),$ $\mathrm{Rad}(L)$, $\mathrm{Prim}(L)$ are all compact.
\end{corollary}

While Proposition \ref{csb} presents a sufficient condition for a multiplicative lattice space to possess compactness, the subsequent proposition delivers a characterization of these spaces. However, before proceeding with the proposition, we need  a lemma.

\begin{lemma}\label{upper-bound}
If $X$ is a compact, T$_0$-space, then every chain in $X$ has an upper bound. 
\end{lemma}

\begin{proposition}\label{cqc}
Suppose $L$ is a compactly generated multiplicative lattice with $1$ compact, and let $\Sigma_L$ be a class of elements of $L$. Suppose $\mathrm{Max}(\Sigma_L)$ denotes the set of maximal elements of $\Sigma_L$. Then the following are equivalent: 
\begin{enumerate}
\item $\Sigma_L$ is compact;
		
\item for every $x\in\Sigma_L,$ there exists an $m\in\mathrm{Max}(\Sigma_L)$ such that $x\leqslant m$, and $\mathrm{Max}(\Sigma_L)$ is compact.
\end{enumerate}
\end{proposition}

\begin{proof}
To show 2.$\Rightarrow$1., let $\mathcal{U}$ be an open cover of $\Sigma_L$. Then, $\mathcal{U}$ also covers $\mathrm{Max}(\Sigma_L)$, and thus there is a finite subcover $\mathcal{U}'$ of $\mathrm{Max}(\Sigma_L)$. Consider $x\in\Sigma_L$. By hypothesis, there is $m\in\mathrm{Max}(\Sigma_L)$ such that $x\leqslant m$, and $U\in\mathcal{U}'$ such that $m\in U$. Then, $x\in U$, and thus $\mathcal{U}'$ is also a finite subcover for $\Sigma_L$. Hence, $\Sigma_L$ is compact.

Next we prove that 1.$\Rightarrow$2. Suppose that $x\in\Sigma_L$, and consider \[\Sigma_L':=v(x)\wedge \Sigma_L.\] Then, $\Sigma_L'$ is a closed subset of $\Sigma_L$, and thus it is compact; by Lemma \ref{upper-bound}, every ascending chain in $\Sigma_L'$ is bounded, and hence by Zorn's Lemma $\Sigma_L'$ has maximal elements, that is, $x\leqslant m,$ for some $m\in\mathrm{Max}(\Sigma_L)$.
Now, suppose $\mathcal{U}$ is an open cover of $\mathrm{Max}(\Sigma_L)$. Then, $\mathcal{U}$ is also an open cover of $\Sigma_L$, and thus it admits a finite subcover, which will be also a finite subcover of $\mathrm{Max}(\Sigma_L)$. Thus, $\mathrm{Max}(\Sigma_L)$ is compact.
\end{proof}

Next we investigate the lower separation axioms of lower spaces, specifically $T_0$, $T_1$, and sobriety. Notably, sobriety will hold significant importance as it aids in characterizing lower spaces that are spectral.

\begin{proposition}\label{t0a}
Every lower space of a multiplicative lattice is $T_0$. 
\end{proposition} 

\begin{proof}
Let $x,$ $y\in \Sigma_L$ such that $v(x)=v(y).$  Since $y\in v(x),$ we have $x\leqslant y$. Similarly, we obtain $y\leqslant x$. Hence, $x=y$, and we have the claim.  	
\end{proof}

The following auxiliary result is going to be useful in further investigation of separation axioms as well as in studying connectedness of lower spaces.

\begin{lemma}\label{irrc}
Every subbasic closed set of the form $\left\{v(x)\mid x\in v(x)\right\}$ of a lower space $\Sigma_L$ is irreducible. 
\end{lemma} 

\begin{proof} 
We show more, namely, $v(x)=\cl(x)$ for every $x$ such that $x\in v(x)$. By definition of $\cl(\cdot)$ and hypothesis,	 we have  $\cl(x)\subseteq v(x)$. 
We obtain the other inclusion by considering the following two cases. Suppose
$\cl(x)= \Sigma_L$. Then  we have the claim from the fact that
\[
\Sigma_L=\cl(x)\subseteq v(x)\subseteq \Sigma_L.
\]
Now suppose $\cl(x)\neq \Sigma_L$. Since $\cl(x)$ is a closed set, there exists an index set $\Omega$ such that for each $\omega\in\Omega$, there is a positive integer $n_{\omega}$ and $x_{\omega 1},\dots, x_{\omega n_\omega}\in L$ such that 
\[
\cl(x)={\bigwedge_{\omega\in\Omega}}\left({\bigvee_{ i\,=1}^{ n_\omega}}v(x_{\omega i})\right).
	\]	
Since 
$\cl(x)\neq \Sigma_L$, we may assume that  ${\bigvee_{ i\,=1}^{ n_\omega}}v(x_{\omega i})$ is non-empty for each $\omega$. Therefore,  $x\in   {\bigvee_{ i\,=1}^{ n_\omega}}v(x_{\omega i})$, for each $\omega$, and from which we can deduce that $v(x)\leqslant {\bigvee_{ i=1}^{ n_\omega}}v(x_{\omega i})$, that is, $v(x)\leqslant \cl(x)$. 	
\end{proof}

\begin{corollary}\label{spiir}
Every non-empty subbasic closed set of the lower space of proper elements is irreducible.
\end{corollary}

\begin{theorem}\label{ZariskiT1}
Let $L$ be a compactly generated multiplicative lattice with $1$ compact.
A lower space $\Sigma_L$ is a $T_1$ if and only if $m\in \Sigma_L$ for all $m\in \mathrm{Max}(L)$.
\end{theorem}

\begin{proof}
Suppose that $x\in \Sigma_L$. Then $x\in v(x)$, and so, by Corollary \ref{spiir}, $\cl(x)=v(x)$. Let $m$ be a maximal element of $L$ such that $x\leqslant m$. Then   \[m\in 	v(x)=\cl(x) = \{x\},\] where the second equality follows from the fact that $\Sigma_L$ is a $T_{ 1}$-space. Therefore $m=x$, showing that $\Sigma_L\subseteq \mathrm{Max}(L)$.  
Conversely,  $v(m)=\{m\}$ holds for every $m\in \mathrm{Max}(L)$, so that $m\in v(m)$. Hence, $\cl(m)=\{m\}$, by Corollary \ref{spiir}, and this proves the claim.
\end{proof}

Next we characterize sobar lower spaces of multiplicative lattices.

\begin{theorem}
\label{sob}  
An lower space $\Sigma_L$ is sober if and only if
\begin{equation}\label{soc}
\bigwedge_{\substack{x\leqslant y\\ y\in \Sigma_L}}  v(y)\subseteq v(x),
\end{equation}
for every non-empty irreducible subbasic closed subset of $v(x)$ of $\Sigma_L$.
\end{theorem}

\begin{proof}
If $\Sigma_L$ is a sober space and $v(x)$ is a non-empty irreducible subbasic closed subset of $\Sigma_L$, then  
$v(x)=\cl(\{y\})=v(y)$
for some $y\in \Sigma_L$, and we have \[\displaystyle y=\bigwedge_{\substack{x\leqslant y\\ y\in \Sigma_L}}  v(y)\in \Sigma_L.\] 
Conversely, suppose the condition (\ref{soc}) holds for every non-empty irreducible subset of $\Sigma_L$. Let $C$ be an irreducible closed subset of $\Sigma_L$. Then \[C=\bigwedge_{i\in \Omega}\left (  \bigvee_{j=1}^m v(y_{ji}) \right),\] for some elements $y_{ji}$ of $L$. Since $C$ is irreducible, for every $i\in \Omega$ there exists an element $y_{ji}$ of $L$ such that $C\subseteq v(y_{ji})\subseteq \bigvee_{j=1}^mv(y_{ji})$ and thus, if $t= \bigvee_{i\in \Omega}y_{ji}$, then we have  \[C=\bigwedge_{i\in \Omega}v(y_{ji})=v(t)=v\left(\bigwedge_{\substack{t\leqslant l \\ l\in \Sigma_L}}  l \right)=\cl\left({\bigwedge_{\substack{t\leqslant l\\ l\in \Sigma_L}}  l}\right).\]
The uniqueness part of the claim follows from Proposition \ref{t0a}.
\end{proof}

\begin{corollary}\label{qss}
The lower spaces $\mathrm{Prop}(L)$, $\mathrm{Spec}(L)$, and $\mathrm{Irr}^+(L)$ are sober.
\end{corollary}

Our subsequent objective is to furnish an example of a spectral lower space. Recall from \cite{Hoc69} that a topological space is called \emph{spectral} if it is compact, sober,  admitting a basis of compact open subspaces that is closed under finite intersections.

\begin{theorem}\label{tsqs}
Let $L$ be a compactly generated multiplicative lattice with $1$ compact, and satisfies the ``max-bounded'' condition. Then
the lower space $\mathrm{Prop}(L)$ is spectral;
\end{theorem}

The assertion is derived through the application of the following lemma, the proof of which  can be found in \cite[Lemma 2.2]{Gos23}. 	
	
\begin{lemma}\label{cso}
A compact, sober, open subspace of a spectral space is spectral. 
\end{lemma}
	
The advantage of the aforementioned lemma lies in its capacity to circumvent the need to establish the existence of a compact open basis and the closedness of open compact sets under finite intersections.

\begin{proof}	
The lower space  of all elements is spectral follows from the fact that an algebraic lattice endowed with lower topology is a spectral space (see \cite[Theorem 4.2]{Pri94}).
Since the lower space $\mathrm{Prop}(L)$ is a subspace of the lower space $L$, according to Lemma \ref{cso}, we need to verify that  lower space $\mathrm{Prop}(L)$ is compact, sober, and
is an open subspace of the lower space $L$. Now compactness and sobriety respectively follow from Corollary \ref{ciqc} and Corollary \ref{qss}. 
Therefore, what remains is  to show that  
the lower space  is open in all elements.
By Lemma \ref{irrc}, \[L=\bigvee_{l\in L}v(l)=\cl(L),\] and therefore $L \backslash \mathrm{Prop}(L) $ is closed, and that implies $\mathrm{Prop}(L)$ is open. 
\end{proof} 

Like Alexander subbasis theorem, there is no characterization of connectedness in terms of subbasic closed sets. Nevertheless, we wish to present a disconnectivity result of lower spaces of a multiplicative lattice that bears resemblance to the fact that if spectrum of prime ideals (of a commutative ring with identity) endowed with Zariski topology is disconnected, then the ring has a proper idempotent element (see \cite[\textsection 4.3, Corollary 2]{Bou72}). 

We say a closed subbasis $\mathfrak{S}$ of a topological space $X$ \emph{strongly disconnects} $X$  if there exist two non-empty subsets $A,$ $B$ of $\mathfrak{S}$ such that \[X=A\vee  B\quad  \text{and}\quad  A\wedge  B=\emptyset.\]
It is clear that if some closed subbasis strongly disconnects  a  topological space, then the space is disconnected. Also, if a space is disconnected, then some closed subbasis (for instance the collection of all its closed subspaces) strongly disconnects it.

\begin{proposition}\label{pr1}  
Suppose that $L$ is a multiplicative lattice with  $\mathrm{Jac}(L)=0$. Let $\Sigma_L$ be a class of elements of $L$ containing all maximal elements of $L.$ If the subbasis $\mathfrak{S}$ of the lower space $\Sigma_L$ strongly  disconnects  $\Sigma_L$, then $L$ has a non-trivial idempotent element.
\end{proposition} 

\begin{proof}
Let $x$ and $y$ be elements in $L$ such that
1. $
v(x)\wedge  v(y) =\emptyset,$
2. $v(x)\vee  v(y) =\Sigma_L,$ and  
3. $v(x)\ne\emptyset, $ $ v(y) \neq \emptyset. $
Since $v(x)\wedge  v(y)=v(x \vee y)$, we therefore have $v(x\vee y)=\emptyset$ and hence $x\vee y =1$, because $\Sigma_L$ contains all maximal elements in $L$. On the other hand, 
\[
\Sigma_L=v(x)\vee  v(y)\subseteq v(x\cdot y),
\] which then implies that $x\cdot y$ is contained in every maximal element in $L$, and is therefore $x\cdot y =0$  as $L$ has trivial $\mathrm{Jac}(L)$. Note that the condition 3. implies that neither $x$ nor $y$ is the top element $1$ of the multiplicative lattice $L$.  Since  $x\cdot y=0$, we therefore have
\[
x=x\cdot 1= x\cdot(x\vee y)=x^{ 2}\vee (x\cdot y) =x^2\vee 0=x^2,
\]
showing that $x$ is a non-trivial  idempotent element in $L$. 
\end{proof}

The next theorem gives a sufficient condition for a lower space to be connected.

\begin{theorem}\label{conn}
If a class $\Sigma_L$ of elements of a multiplicative lattice $L$ contains the element $0$, then the lower space $\Sigma_L$ is connected. 
\end{theorem}

\begin{proof}
Since  $\Sigma_L=v(0)$ and irreducibility implies connectedness, the claim follows from Corollary \ref{spiir}.
\end{proof}

\begin{corollary}
The lower space $\mathrm{Prop}(L)$ is connected. 
\end{corollary}

We now discuss about continuous maps between lower spaces of multiplicative lattices. Recall that a map $\phi\colon L \to L'$ from $L$ to $L'$ is called a \emph{multiplicative lattice homomorphism} if 
\begin{enumerate}
\item  $x\leqslant  x'$ implies that $\phi(x)\leqslant  \phi(x');$	
	
\item  $\phi(x\vee x')=\phi(x)\vee \phi(x');$
	
\item  $\phi(x\wedge x')=\phi(x)\wedge \phi(x');$
	
\item  $\phi(x\cdot  x')=\phi(x)\cdot  \phi(x'),$

\item $\phi(1)=1$,
\end{enumerate}
for all $x,$ $x'\in L$. 
Suppose that  $\phi\colon L \to L'$ is a multiplicative lattice homomorphism.
If $y$ is an element in $L'$, then the \emph{contraction of} $y$ is defined by $\phi\inv (y).$ In particular, the \emph{kernel of} $\phi$ is defined as $\mathrm{Ker}\phi:=\phi\inv(0)$.

Observe that although inverse image of an element from a given class $\Sigma_{L'}$ under a multiplicative lattice homomorphism  may not belong to the similar class $\Sigma_L$. 
We say a class $\Sigma_L$ of elements satisfies the \emph{contraction} property if for any multiplicative lattice homomorphism $\phi\colon L\to L',$ the inverse image  $\phi\inv(y)$ is in $\Sigma_L$, whenever $y$ is in $\Sigma_{L'}.$

\begin{proposition}\label{conmap}
Let $L$ be a multiplicative lattice. Suppose that $\Sigma_L$ is a class of elements satisfying the contraction property. Let $\phi\colon L\to L'$ be a multiplicative lattice homomorphism  and $y\in\Sigma_{L'}.$ 
\begin{enumerate}
		
\item\label{contxr} The induced map $\phi_*\colon  \Sigma_{L'}\to \Sigma_L$ defined by  $\phi_*(y)=\phi\inv(y)$ is    continuous.
		
\item If $\phi$ is  surjective, then the multiplicative lattice space $\Sigma_{L'}$ is homeomorphic to the closed subspace \[\mathrm{Ker}\phi^{\uparrow}:=\{ x\in \Sigma_L\mid k\leqslant x,\forall k\in \mathrm{Ker}\phi \}\] of the multiplicative lattice space $\Sigma_L.$	
\item\label{den} The subset  $\phi_*(\Sigma_{L'})$ is dense in $\Sigma_L$ if and only if  $k\leqslant  \bigwedge_{x\in \Sigma_L}x,$ for every $k\in\mathrm{Ker}\phi.$ 
\end{enumerate}
\end{proposition}

\begin{proof}      
To show 1., let $x\in L$ and $v(x)$ be a   subbasic closed set of the lower space $\Sigma_L.$ Then  
\begin{align*}
\phi_*\inv(v(x)) &=\left\{ y\in  \Sigma_{L'}\mid \phi\inv(y)\in v(x)\right\}\\&=\left\{y\in \Sigma_{L'}\mid \phi(x)\leqslant  y\right\}=v(\phi(x)), 
\end{align*} 
and hence the map $\phi_*$  continuous.  
	
2. Observe that $\mathrm{Ker}\phi\subseteq  \phi\inv(y)$ follows from the fact that  $0\leqslant  y$ for all $y\in \Sigma_{L'}.$ It can thus been seen that $\phi_*(y)\in \mathrm{Ker}\phi^{\uparrow},$ and hence $\mathrm{Im}\phi_*=\mathrm{Ker}\phi^{\uparrow}.$  
If $y\in \Sigma_{L'},$ then
\[\phi\left(\phi_*\left(y\right)\right)=\phi\left(\phi\inv\left(y\right)\right)=y.\]
Thus $\phi_*$ is injective. To show that $\phi_*$ is closed, first we observe that for any   subbasic closed set  $v(x)$ of  $\Sigma_{L'}$, we have
\begin{align*}
\phi_*\left(v(x)\right)&=  \phi\inv\left(v(x)\right)\\&=\phi\inv\left\{ y\in \Sigma_{L'}\mid x\leqslant    y\right\}=v(\phi\inv(x)). 
\end{align*}
Now if $C$ is a closed subset of $\Sigma_{L'}$ and $C=\bigwedge_{ \omega \in \Omega} \left(\bigvee_{ i \,= 1}^{ n_{\omega}} v(x_{ i\omega})\right),$ then
\begin{align*}
\phi_*(C)&=\phi\inv \left(\bigwedge_{ \omega \in \Omega} \left(\bigvee_{ i = 1}^{ n_{\omega}} v(x_{ i\omega})\right)\right)\\&=\bigwedge_{ \omega \in \Omega}\, \bigvee_{ i = 1}^{n_{\omega}} \phi\inv\left(v(x_{ i\omega})\right)
\end{align*}
a closed subset of  $\Sigma_L.$ Since by (\ref{contxr}), $\phi_*$ is continuous, we have the desired claim.
	
3. First we wish to show: \[\cl\left(\phi_*\left(v(y)\right)\right)=\phi\inv\left(v(y)\right),\] for all elements $y$ of $L'$. For that, let $z\in \phi_*\left(v(y)\right).$ This implies $\phi\left(z\right)\in v(y),$ and that means $y\leqslant  \phi(z).$ Therefore, $z\in v(\phi\inv\left(y\right)).$ Since $\phi_*\left(v(y)\right)=v(\phi\inv\left(y\right))$, the other inclusion follows. If we take $y=0'\in L'$, the above identity reduces to 
\[\cl\left(\phi_*\left(\Sigma_{L'}\right)\right)=\mathrm{Ker}\phi^{\uparrow},\] and hence  $\mathrm{Ker}\phi^{\uparrow}$ to be equal to $\Sigma_L$ if and only if the desired condition holds.
\end{proof}  

We finally wish to see relations between  the lower space $\mathrm{Irr}^+(L)$ and its subspaces $\mathrm{Max}(L)$ and   $\mathrm{Spec}(L)$. For that, we first introduce a couple of terminology. The \emph{$p$-radical} $\sqrt[p]{L}$ (respectively \emph{$s$-radical} $\sqrt[s]{L}$) of a multiplicative lattice $L$ is the intersection of all prime (respectively strongly irreducible) elements of $L$.

\begin{proposition}
Let $L$ be a compactly generated multiplicative lattice with $1$ compact.
\begin{enumerate}
		
\item The subspace $\mathrm{Max}(L)$ is dense in the lower space $\mathrm{Irr}^+(L)$ if and only if $\sqrt[p]{L}=\sqrt[s]{L}.$
		
\item The subspace $\mathrm{Spec}(L)$  is dense in the lower space $\mathrm{Irr}^+(L)$ if and only if $\mathrm{Jac}(L)=\sqrt[s]{L}.$
\end{enumerate}
\end{proposition}

\begin{proof}
1. Although the claim essentially follows from the fact that if $X\subseteq \mathrm{Irr}^+(L)$, then 
\[\cl(X)=\left\{y\in \mathrm{Irr}^+(L)\mid  \bigwedge_{x\in X}x\leqslant y\right\},\] however, we provide some details.  Let $\cl(\mathrm{Spec}(L))=\mathrm{Irr}^+(L).$ Then \[\left\{y\in \mathrm{Irr}^+(L)\mid \bigwedge_{p\in \mathrm{Spec}(L)}p\leqslant y\right\}=\mathrm{Irr}^+(L).\] This implies that
$\sqrt[p]{L}\subseteq \sqrt[s]{L}.$
Furthermore, $\mathrm{Max}(L)\subseteq \mathrm{Irr}^+(L)$ implies $\sqrt[s]{L}\subseteq \sqrt[p]{L}$. Hence, we have the desired equality. To obtain the converse, let $\mathrm{Irr}^+(L)\setminus \cl(\mathrm{Spec}(L))\neq \emptyset.$ This implies $y\notin \cl(\mathrm{Spec}(L))$, but $y\in \mathrm{Irr}^+(L).$ Therefore, there exists a neighbourhood $N_y$ of $y$ such that $N_y\cap \mathrm{Spec}(L)=\emptyset,$ and $\sqrt[s]{L}\subsetneq \sqrt[p]{L}.$ In other words, we have  $\sqrt[s]{L}\neq \sqrt[p]{L}.$
	
2. Follows from 1.
\end{proof}

\end{document}